\documentclass[10pt, english]{article}
\usepackage[T1]{fontenc}
\usepackage[utf8]{inputenc}
\usepackage{lmodern}
\usepackage[a4paper]{geometry}

\usepackage{amsmath}
\usepackage{amsfonts}
\usepackage{amssymb}
\usepackage{amsthm}

\usepackage{mathtools}
\usepackage{stmaryrd}
\usepackage{systeme}
\usepackage{tdsfrmath}
\usepackage{cases}
\usepackage{eqnarray}
\usepackage{nicematrix}
\usepackage{empheq}

\usepackage{csquotes}
\usepackage{parskip}
\usepackage{enumitem}

\usepackage{chngcntr}

\usepackage{tikz}
\usetikzlibrary{arrows.meta}

\usepackage{nth}

\usepackage[all]{xy}

\usepackage{babel}

\usepackage[colorlinks=true]{hyperref}
\usepackage{cleveref}

\usepackage{bookmark}

\usepackage{glossaries}
\usepackage[maxbibnames=9, maxnames=4, url=true]{biblatex}

\numberwithin{equation}{section}

\theoremstyle{plain}
\newtheorem{theo}{Theorem}[section]

\newtheorem{lem}[theo]{Lemma}

\theoremstyle{definition}
\newtheorem{defi}[theo]{Definition}

\theoremstyle{remark}
\newtheorem{rema}[theo]{Remark}

\setenumerate[1]{font=\normalfont, label={(\roman{enumi})}}

\allowdisplaybreaks[3]

\title{A criterion of contractivity for $4 \times 4$ upper-triangular matrices}
\author{Axel \textsc{Renard}}
\date{}

\NewDocumentCommand{\cj}{m}{\overline{#1}}
\NewDocumentCommand{\abs}{m}{\left|#1\right|^2}
\NewDocumentCommand{\ps}{m}{\langle #1 \rangle}
\NewDocumentCommand{\id}{}{\ensuremath{\mathrm{Id}}}
\NewDocumentCommand{\oud}{}{\ensuremath{\mathbb{D}}}
\NewDocumentCommand{\cud}{}{\ensuremath{\overline{\mathbb{D}}}}

\NewDocumentCommand{\eg}{}{\textit{e.g. }}

\makeatletter
\renewcommand\maketitle
{\begin{center}\noindent
		{\Large\bfseries\@title}%
		\bigskip\par\noindent
		{\large\@author}%
		\medskip\par\noindent
		{\large\@date}%
		\bigskip\par\noindent
\end{center}}
\makeatother

\AtEndDocument{\bigskip{\footnotesize%
		(A. Renard) \textsc{Univ Lille, CNRS, UMR 8524 - Laboratoire Paul Painlev\'e, France}
		\par  
		\textit{E-mail address}: \texttt{axel.renard@univ-lille.fr}
}
}

\begin{document}

\maketitle

\begin{abstract}
		\noindent We establish an explicit criterion for determining whether a $4 \times 4$ upper-triangular matrix is a contraction with respect to the Euclidean operator norm. 
\end{abstract}

\section*{Notation}
We start by introducing some notation. In this paper, we will denote by $\oud$ the open unit disk, and by $\mathbb{T}$ the unit circle, \textit{i.e.} $\mathbb{T} = \cud \backslash \oud$. For $n \in \mathbb{N}^*$, we will also denote by $\mathcal{M}_n(\mathbb{C})$ the set of $n \times n$ matrices with complex coefficients, and by $\|\cdot\|$ the Euclidean operator norm on $\mathbb{C}^n$. Throughout this paper, we operate within the framework of Hilbert spaces, considering
$\left(\mathbb{C}^n, \ps{\cdot, \cdot} \right)$ as an $n$-dimensional Hilbert space, where $ \ps{\cdot, \cdot}$ denotes the standard inner product on $\mathbb{C}^n$. Moreover, if $\|T\| \le 1$, we say that $T$ is a contraction. The operator $T$ is a contraction if and only if $\id - T^*T$ is semi-positive definite. Here $\id$ is the identity operator and $T^*$ is the conjugate transpose of the matrix $T$, \textit{i.e.} the (Hilbertian) adjoint of $T$. When $T$ is a contraction, we will denote by $D_T$ its defect operator defined by $D_T=(\id-T^*T)^{1/2}$, where the \emph{square root} denotes the unique semi-positive definite square root. Finally, in this manuscript, $\llbracket N, M \rrbracket$ denotes all the integers from $N$ to $M$.

\section{Introduction}
The calculation of the (Euclidean) operator norm of a general $2\times 2$ matrix $T$ can be explicitly determined using the fact that $\|T\|$ is the largest singular value of $T$. However, computing $\|T\|$ explicitly for matrices $T$ of larger sizes becomes more complex and does not yield very explicit criteria. In \autocite{badeaSchwarzPickTypeInequalities2024b}, the following criterion for a $3 \times 3$ upper-triangular matrix to be a contraction has been established, following the idea of \citeauthor{guptaCaratheodoryFejerInterpolationProblems2015} in \autocite[Lemma 2.7]{guptaCaratheodoryFejerInterpolationProblems2015}:

\begin{theo}\label{gupta}
	Let $\omega_1, \omega_2, \omega_3 \in \mathbb{D}$. Then, $T_3 = \begin{pmatrix}\omega_1 & \alpha_1 & \beta \\ 0 & \omega_2 & \alpha_2 \\ 0 & 0 & \omega_3\end{pmatrix} \in \mathcal{M}_3(\mathbb{C})$ is a contraction if and only if:
	\begin{empheq}[left= \qquad \empheqlbrace]{flalign}
		& \quad |\omega_2| < 1 & \\
		& \quad \left|\alpha_i\right|^2 \leq (1-|\omega_i|^2)(1-|\omega_{i+1}|^2)\label{cond-3-times-3-1}, \: i=1,2 & \\
		& \quad \left|\beta (1-|\omega_2|^2)+\alpha_1\alpha_2\overline{\omega_2}|^2\right| \leq \nonumber & \\
		& \quad \left[(1-|\omega_1|^2)(1-|\omega_2|^2) - |\alpha_1|^2\right] \cdot \left[(1-|\omega_2|^2)(1-|\omega_3|^2) - |\alpha_2|^2 \right]\label{cond-3-times-3-2} &
	\end{empheq}
	\begin{empheq}[left= \text{ or }~ \empheqlbrace]{flalign}
		&\quad |\omega_2| =1 \nonumber & \\
		& \quad \alpha_i =0, \: i=1,2 \nonumber & \\
		& \quad |\beta|^2  \leq (1-|\omega_1|^2)(1-|\omega_3|^2) \nonumber& 
	\end{empheq}
\end{theo}

In this paper, we will establish the following explicit criterion for determining whether a $4 \times 4$ upper-triangular matrix is a contraction:

\begin{theo}\label{critere_4_4_gen}
	Let $T= \begin{pmatrix}
		\omega_1 & \alpha_1 & \beta_1 & \gamma \\
		0 & \omega_2 & \alpha_2 & \beta_2 \\
		0 & 0 & \omega_3 & \alpha_3 \\
		0 & 0 & 0 & \omega_4 
	\end{pmatrix} \in \mathcal{M}_4(\mathbb{C})$, with $\omega_1, \omega_2, \omega_3, \omega_4 \in \cud$.
	
	Then $T$ is a contraction when acting on the Hilbert space $\C^4$ if and only if: 
	
	\begin{empheq}[left= \quad \empheqlbrace]{flalign}
		& |\omega_2| < 1 \quad ; \qquad |\omega_3| <1 \nonumber & \\
		& |\alpha_i|^2 \leq (1-|\omega_i|^2)(1-|\omega_{i+1}|^2), \quad i=1,3 \label{contrac_4_4_premier_ordre_gen} & \\
		& |\alpha_2|^2 < (1-|\omega_2|^2)(1-|\omega_3|^2) \nonumber & \\
		& \begin{multlined}\label{contrac_4_4_deuxieme_ordre_gen}|\beta_i (1-|\omega_{i+1}|^2)+ \alpha_i \alpha_{i+1} \cj{\omega_{i+1}} |^2 \leq \left[(1-|\omega_i|^2)(1-|\omega_{i+1}|^2) - |\alpha_i|^2\right] \times \\ \left[(1-|\omega_{i+1}|^2)(1-|\omega_{i+2}|^2) - |\alpha_{i+1}|^2 \right], \qquad i=1,2 \end{multlined} & \\
		& \begin{multlined}\label{contrac_4_4_troisieme_ordre_gen}
			\left| \gamma [(1-|\omega_2|^2)(1-|\omega_3|^2)-|\alpha_2|^2]+\alpha_1 \beta_2 \cj{\omega_2}(1-|\omega_3|^2) + \alpha_3 \beta_1 \cj{\omega_3}(1-|\omega_2|^2) \right. \\ \left. +\beta_1\beta_2\cj{\alpha_2} + \alpha_1\alpha_2\alpha_3 \cj{\omega_2 \omega_3} \right|^2 (1-|\omega_2|^2)(1-|\omega_3|^2) \\
			\leq \left[ \left( (1-|\omega_1|^2)(1-|\omega_2|^2) - |\alpha_1|^2 \right) \left( (1-|\omega_2|^2)(1-|\omega_3|^2) - |\alpha_2|^2 \right) - \right. \\ \left. \left| \alpha_1 \alpha_2 \cj{\omega_2} + \beta_1 (1-|\omega_2|^2) \right|^2 \right] \\
			\times \left[ \left( (1-|\omega_2|^2)(1-|\omega_3|^2) - |\alpha_2|^2 \right) \left( (1-|\omega_3|^2)(1-|\omega_4|^2) - |\alpha_3|^2 \right) \right. \\ \left.- \left| \alpha_2 \alpha_3 \cj{\omega_3} + \beta_2 (1-|\omega_3|^2) \right|^2 \right]
		\end{multlined} &
	\end{empheq}

	\begin{empheq}[left= \text{or } \empheqlbrace]{flalign}
		& |\omega_2| < 1 \quad ; \quad |\omega_3| <1 \nonumber & \\
		& |\alpha_i|^2 \leq (1-|\omega_i|^2)(1-|\omega_{i+1}|^2), \qquad i=1,3 \label{contrac_4_4_lim_premier_ordre_gen}& \\
		& |\alpha_2|^2 = (1-|\omega_2|^2)(1-|\omega_3|^2)  \nonumber& \\
		& \beta_i = \frac{-\alpha_i \alpha_{i+1} \omega_{i+1}}{1 - |\omega_{i+1}|^2}, \qquad i=1,2 \label{contrac_4_4_lim_deuxieme_ordre_gen}& \\
		& \begin{multlined}\left| \gamma - \frac{\cj{\omega_2} \cj{\omega_3} \alpha_1 \alpha_2 \alpha_3}{(1-|\omega_2|^2)(1-|\omega_3|^2)} \right|^2 \left(1-|\omega_2|^2\right)\left(1-|\omega_3|^2\right) \\ \leq  \left[(1-|\omega_1|^2)(1-|\omega_2|^2) - |\alpha_1|^2 \right] \times \left[ (1-|\omega_3|^2)(1-|\omega_4|^2) - |\alpha_3|^2 \right] \end{multlined}\label{contrac_4_4_lim_troisieme_ordre_gen}&
	\end{empheq}

	or~~$\begin{cases} |\omega_2| = 1 \quad ; \quad |\omega_3| < 1 \\
		\alpha_1 = \alpha_2 = 0 \\
		|\alpha_3|^2 \leq (1-|\omega_3|^2)(1-|\omega_4|^2) \\
		|\beta_1 |^2 \leq (1-|\omega_1|^2)(1-|\omega_3|^2) \\
		\left|\beta_2\left(1-|\omega_3|^2\right) \right|^2 \leq (1-|\omega_3|^2)(1-|\omega_4|^2) - |\alpha_3|^2 \\
		\left|\gamma (1- |\omega_3|^2) +  \alpha_3 \beta_1 \cj{\omega_3} \right|^2 \leq \left[ (1-|\omega_3|^2)(1-|\omega_4|^2) - |\alpha_3|^2\right] \left[(1-|\omega_1|^2)(1-|\omega_3|^2) - |\beta_1|^2 \right]
	\end{cases}$
	
	or~~$\begin{cases} |\omega_3| = 1 \quad ; \quad |\omega_2| < 1 \\
		|\alpha_1|^2 \leq (1-|\omega_1|^2)(1-|\omega_2|^2) \\
		\alpha_2 = \alpha_3= 0 \\
		|\beta_1(1-|\omega_2|^2)|^2 \leq (1-|\omega_1|^2)(1-|\omega_2|^2)-|\alpha_1|^2 \\
		|\beta_2|^2 \leq (1-|\omega_2|^2)(1-|\omega_4|^2) \\
		\left|\gamma (1-|\omega_2|^2) + \alpha_1 \beta_2 \cj{\omega_2} \right|^2 \leq  \left[ (1 - |\omega_1|^2)(1-|\omega_2|^2) - |\alpha_1|^2 \right] \left[(1-|\omega_2|^2) (1 - |\omega_4|^2) - |\beta_2|^2  \right]
	\end{cases}$
	
	or~~$\begin{cases} |\omega_2| = |\omega_3| = 1 \\
		\alpha_1 = \alpha_2 = \alpha_3 = 0 \\
		|\beta_i|^2 \leq (1-|\omega_i|^2)(1-|\omega_{i+2}|^2), \qquad i=1,2 \\
		\left|\gamma \right|^2 \leq (1-|\omega_1|^2)(1-|\omega_4|^2)
	\end{cases}$
\end{theo}

In order to do this, we will follow the same approach as for $3 \times 3$ matrices, employing a slightly more precise version of Parrott's theorem on matrix completion (see \autocite{parrottQuotientNormSzNagyth1978, foiasCommutantLiftingApproach1990a}, \autocite[Theorem 12.22]{youngIntroductionHilbertSpace1988b} or \autocite{familyarseneCompletingMatrixContractions1982, davisNormPreservingDilationsTheir1982b} for the usual version of the theorem, and \autocite[Appendix]{badeaSchwarzPickTypeInequalities2024b} for a revisited proof of the theorem):
\begin{theo}[Parrott]\label{parrott}
	Let $H_1, H_2, K_1, K_2$ be Hilbert spaces, and assume that the operators $\begin{bmatrix}
		A \\ C
	\end{bmatrix}\in\mathcal{B}(H_1,K_1 \oplus K_2)$ and $\begin{bmatrix}
		C & D
	\end{bmatrix} \in \mathcal{B}(H_1 \oplus H_2,K_2)$ are contractions.
	
	Then, \[T= \begin{bmatrix}
		A & B \\ C & D
	\end{bmatrix} : H_1  \oplus H_2 \to K_1 \oplus K_2\] is a contraction if and only if there exists a contraction $W \in \mathcal{B}(H_2,K_1)$ such that: 
	\[
	B=D_{Z^*}WD_Y - Z C^*Y,
	\] where $Z \in \mathcal{B}(H_1, K_1)$ and $Y \in \mathcal{B}(H_2,K_2)$ are contractions such that $D=D_{C^*}Y$ and $A=ZD_C$. 
	
	Moreover, 
	\begin{enumerate}
		\item $Y$ and $Z$ can be chosen to be (respectively) $Y_0$ and $Z_0$, the solutions of minimal operator norm among all solutions of the operator equations $D=D_{C^*}Y$ and  $A=ZD_C$;
		\item If $T$ is a contraction, there exists a unique contraction $W_0$ such that: 
		\[
		B=D_{Z_0^*}W_0D_{Y_0} - Z_0 C^*Y_0 \; \text{ and } \; \text{Im}\left(D_{Z_0^*}\right)^{\perp} \subset \text{Ker}(W_0^*). 
		\]
		This operator satisfies: 
		\[ 
		\|W_0\| = \inf\{ \, \|W\| \, : \, B=D_{Z_0^*}WD_{Y_0} - Z_0 C^*Y_0 \}. 
		\]
	\end{enumerate} 
	We shall call $Y_0$ and $Z_0$ the \emph{minimal solutions} and we shall refer to $W_0$ as the \emph{minimal solution} of the equation 
	\[
	B=D_{Z_0^*}WD_{Y_0} - Z_0 C^*Y_0.
	\]
\end{theo}

The properties of the minimal solutions will be crucial in the proof of \Cref{critere_4_4_gen}.

\section{Proof of \Cref{critere_4_4_gen}}

First of all, we will need the following technical lemma, which provides the diagonalization of a specific matrix $M$. This will allow us to define powers $M^s$ of $M$.

\begin{lem}\label{diag_mat}
	Let $\omega, \alpha \in \mathbb{C}$, $\omega \neq 0$, $M=\begin{pmatrix}
		1-|\omega|^2 & -\overline{\omega}\alpha \\
		- \overline{\alpha}\omega & 1-|\alpha|^2
	\end{pmatrix}$ and $s \in \mathbb{R}$.
	\begin{enumerate}
		\item We have $M=\frac{1}{|\alpha|^2 +|\omega|^2}\begin{pmatrix}
			-\alpha & \overline{\omega} \\
			\omega & \overline{\alpha}
		\end{pmatrix} 
		\begin{pmatrix}
			1 & 0 \\
			0 & 1-|\alpha|^2 - |\omega|^2 
		\end{pmatrix}
		\begin{pmatrix}
			-\overline{\alpha} & \overline{\omega} \\
			\omega & \alpha
		\end{pmatrix}$;
		\item We can define \[M^{s} = \frac{1}{\abs{\alpha}+\abs{\omega}} \begin{pmatrix}
			\abs{\alpha}+\abs{\omega}\left[1 - |\alpha|^2 - |\omega|^2\right]^{s} & -\alpha \cj{\omega}\left(1 - \left[1 - |\alpha|^2 - |\omega|^2\right]^s \right) \\
			-\cj{\alpha}\omega\left(1 - \left[1 - |\alpha|^2 - |\omega|^2\right]^s \right) & |\omega|^2 + |\alpha|^2 \left[1- |\alpha|^2 - |\omega|^2 \right]^s 
		\end{pmatrix},\] provided $\left[1 - |\alpha|^2 - |\omega|^2\right]^s$ is well-defined. 
	\end{enumerate}
\end{lem}
\begin{proof}
This follows from a straightforward computation.
\end{proof}

Now, let us start with the particular case where one of the diagonal element of the matrix, $\omega_3$, is equal to $0$:

\begin{lem}\label{crit-4-4-0}
	Let $T= \begin{pmatrix}
		\omega_1 & \alpha_1 & \beta_1 & \gamma \\
		0 & \omega_2 & \alpha_2 & \beta_2 \\
		0 & 0 & 0 & \alpha_3 \\
		0 & 0 & 0 & \omega_4 
	\end{pmatrix} \in \mathcal{M}_4(\mathbb{C})$, with 
	$\omega_2 \in \oud$ and $\omega_1, \omega_4 \in \cud$.
	
	Then $T$ is a contraction if and only if:
	
	\begin{empheq}[left= \quad \empheqlbrace]{flalign}
		& |\alpha_i|^2 \leq (1-|\omega_i|^2)(1-|\omega_{i+1}|^2), \quad i=1,3 \label{contrac_4_4_lim_cond_prem_ordre} & \\
		&|\alpha_2|^2 = 1-|\omega_2|^2 \nonumber &\\
		&	\beta_1 = \frac{- \alpha_1 \alpha_2 \cj{\omega_2}}{1-|\omega_2|^2} \label{contrac_4_4_lim_cond_deuxieme_ordre_b_1}& \\
		&	\beta_2 =0 \label{contrac_4_4_lim_cond_deuxieme_ordre_b_2} &\\
		&	|\gamma |^2 (1-|\omega_2|^2) \leq  \left[(1- |\omega_1|^2)(1-|\omega_2|^2) - |\alpha_1|^2\right] \left[1-|\omega_4|^2 -|\alpha_3|^2\right] \label{contrac_4_4_lim_cond_trois_ordre}&
	\end{empheq}
	
		
		\begin{empheq}[left= \text{or } \empheqlbrace]{flalign}
			&|\alpha_i|^2 \leq (1-|\omega_i|^2)(1-|\omega_{i+1}|^2), \quad i=1,3 \label{contrac_4_4_premier_ordre} & \\
			&|\alpha_2|^2 < 1-|\omega_2|^2 \nonumber& \\
			&\begin{multlined}
				|\beta_i (1-|\omega_{i+1}|^2)+ \alpha_i \alpha_{i+1} \cj{\omega_{i+1}} |^2 \leq \left[(1-|\omega_i|^2)(1-|\omega_{i+1}|^2) - |\alpha_i|^2\right] \times \\  \left[(1-|\omega_{i+1}|^2)(1-|\omega_{i+2}|^2) - |\alpha_{i+1}|^2 \right], \quad i=1,2 
			\end{multlined}  \label{contrac_4_4_deuxieme_ordre} &\\
			&	\begin{multlined} \left| \gamma \left(1- |\omega_2|^2 - |\alpha_2|^2 \right) + \beta_2 \left(\cj{\omega_2} \alpha_1 + \cj{\alpha_2} \beta_1 \right) \right|^2 (1-|\omega_2|^2)  \\ \leq \left[ \left( (1-|\omega_1|^2)(1-|\omega_2|^2)-|\alpha_1|^2\right) \left(1 - |\omega_2|^2 - |\alpha_2|^2\right) \right. \\ \left. - \left| \beta_1 (1-|\omega_2|^2) + \alpha_1 \alpha_2 \cj{\omega_2} \right|^2 \right] \\ \times \left[ \left(1- |\omega_2|^2 - |\alpha_2|^2\right) \left(1 - |\omega_4|^2 - |\alpha_3|^2  \right) - |\beta_2|^2  \right] \end{multlined} \label{contrac_4_4_troisieme_ordre} &
		\end{empheq}
	\end{lem}
	
	\begin{proof}
		Denote $\omega_3 = 0$. First of all, if $T$ is a contraction, then the compressions $\begin{bmatrix}
			\omega_1 & \alpha_1 & \beta_1 \\
			0 & \omega_2 & \alpha_2 \\
			0 & 0 & \omega_3
		\end{bmatrix}$ and $\begin{bmatrix}
			\omega_2 & \alpha_2 & \beta_2 \\
			0 & \omega_3 & \alpha_3 \\
			0 & 0 & \omega_4 
		\end{bmatrix}$ are contractions. As it is a necessary condition for $T$ to be a contraction, we will assume in the following that it is the case. Thus, from \Cref{gupta}, we have:
		\begin{equation}\label{contrac_4_4_premier_ordre_proof}
			|\alpha_i|^2 \leq (1-|\omega_i|^2)(1-|\omega_{i+1}|^2), \quad 1 \leq i \leq 3
		\end{equation}
		and 
		\begin{equation}
			\begin{multlined}\label{contrac_4_4_deuxieme_ordre_proof}
				|\beta_i (1-|\omega_{i+1}|^2)+ \alpha_i \alpha_{i+1} \cj{\omega_{i+1}} |^2  \leq \left[(1-|\omega_i|^2)(1-|\omega_{i+1}|^2) - |\alpha_i|^2\right] \times \\  \left[(1-|\omega_{i+1}|^2)(1-|\omega_{i+2}|^2) - |\alpha_{i+1}|^2 \right], \quad 1 \leq i \leq 2  
			\end{multlined}
		\end{equation}
		
		Let $A=\begin{bmatrix}
			\omega_1 & \alpha_1 & \beta_1
		\end{bmatrix}$, $B = \begin{bmatrix} \gamma \end{bmatrix}$, $C=\begin{bmatrix}
			0 & \omega_2 & \alpha_2 \\ 0 & 0 & \omega_3 \\ 0 & 0 & 0
		\end{bmatrix}$ and $D=\begin{bmatrix}
			\beta_2 \\ \alpha_3 \\ \omega_4
		\end{bmatrix}$.
		
		By assumption, $\begin{bmatrix}
			A \\ C
		\end{bmatrix}$ and $\begin{bmatrix}
			C & D
		\end{bmatrix}$ are contractions.

		By Parrott's theorem, $T$ is a contraction if and only if
		
		\begin{equation}\label{par-3_preuve_4_4}
			B = (\id-ZZ^*)^{1/2}V(\id-Y^*Y)^{1/2}-ZC^*Y, \text{ for some contraction } V,   \end{equation} 
		
		where $Y$ and $Z$ are contractions such that $D=(\id-CC^*)^{1/2}Y$ and $A=Z(\id-C^*C)^{1/2}$.
		
		Note that the existence of two contractions $Y$ and $Z$ such that $D=(\id-CC^*)^{1/2} Y$ and $A=Z(\id-C^*C)^{1/2}$ is ensured by Parrott's theorem for column (respectively row) matrix-operators, as we are assuming that $\begin{bmatrix}
			A \\ C
		\end{bmatrix}$ and $\begin{bmatrix}
			C & D
		\end{bmatrix}$ are contractions. 
		
		We have
		$\id - C C^* = \begin{bmatrix}
			1 - |\omega_2|^2 - |\alpha_2|^2 & 0 & 0 \\
			0 & 1 & 0 \\
			0 & 0 & 1
		\end{bmatrix}$ \quad \text{and} \quad $\id - C^* C = \begin{bmatrix}
			1 & 0 & 0 \\
			0 & 1- |\omega_2|^2 & -\alpha_2 \cj{\omega_2} \\
			0 & -\omega_2 \cj{\alpha_2} & 1-|\alpha_2|^2
		\end{bmatrix}$.
		
		\begin{itemize}[leftmargin=0pt]
			\item Assume first that $\omega_2 \neq 0$ and that $|\alpha_2|^2 < 1 - |\omega_2|^2$.
			
			Then, denoting $\Sigma = |\alpha_2|^2 + |\omega_2|^2$, we have \[\left(\id - C C^*\right)^{-1/2} = \begin{bmatrix}
				\frac{1}{\sqrt{1-\Sigma}} & 0 & 0 \\
				0 & 1 & 0 \\
				0 & 0 & 1
			\end{bmatrix},\] and, by \Cref{diag_mat}, \[\left(\id - C^* C\right)^{-1/2} = \frac{1}{\Sigma} \begin{bmatrix}
				\Sigma & 0 & 0 \\
				0 & \abs{\alpha_2}+\frac{\abs{\omega_2}}{\sqrt{1 - \Sigma}} & -\alpha_2 \cj{\omega_2}\left[1 - \frac{1}{\sqrt{1 - \Sigma}} \right] \\
				0 & -\cj{\alpha_2}\omega_2\left[1 - \frac{1}{\sqrt{1 - \Sigma}} \right] & |\omega_2|^2 +  \frac{|\alpha_2|^2}{\sqrt{1 - \Sigma}}  
			\end{bmatrix}. \]
			
			Thus, we get $Y = \left(\id - C C^*\right)^{-1/2} D = \begin{bmatrix}
				\frac{\beta_2}{\sqrt{1-\Sigma}} \\ \alpha_3 \\ \omega_4
			\end{bmatrix}$.
				
				Moreover, we get also \begin{align*}
					Z  & = A \left(\id - C^* C\right)^{-1/2} \\
					& = \begin{bmatrix} \omega_1 & \frac{\alpha_1\left(|\alpha_2|^2  \sqrt{1- \Sigma} + |\omega_2|^2\right)-\beta_1 \cj{\alpha_2} \omega_2 \left(\sqrt{1-\Sigma} -1 \right)}{\Sigma \sqrt{1-\Sigma}} & \frac{-\alpha_2\cj{\omega_2}\alpha_1 \left( \sqrt{1- \Sigma} - 1 \right)+\beta_1\left(|\omega_2|^2 \sqrt{1- \Sigma}+|\alpha_2|^2 \right)}{\Sigma \sqrt{1-\Sigma}}\end{bmatrix} \\
					& : = \frac{1}{\Sigma \sqrt{1-\Sigma}} \begin{bmatrix}
						z_1' & z_2' & z_3'
					\end{bmatrix},
				\end{align*} 
				where \begin{align*} 
					& z_1'= \Sigma \sqrt{1-\Sigma} \omega_1; \\
					& z_2'=\alpha_1\left(|\alpha_2|^2  \sqrt{1- \Sigma} + |\omega_2|^2\right)-\beta_1 \cj{\alpha_2} \omega_2 \left(\sqrt{1-\Sigma} -1 \right); \\
					& z_3'=-\alpha_2\cj{\omega_2}\alpha_1 \left( \sqrt{1- \Sigma} - 1 \right)+\beta_1\left(|\omega_2|^2 \sqrt{1- \Sigma}+|\alpha_2|^2 \right).
				\end{align*}
				Thus, we have:
				\[
				ZC^*Y = \frac{\beta_2}{\Sigma \left(1- \Sigma\right)} \left(\cj{\omega_2}z_2' + \cj{\alpha_2}z_3' \right) = \frac{\beta_2}{1-\Sigma} \left(\alpha_1 \cj{\omega_2} + \beta_1 \cj{\alpha_2}\right).
				\]
				
				Now, let us consider the Cholesky factorization (see \textit{e.g.} \autocite[chapter 4]{paulsenIntroductionTheoryReproducing2016b}) of $\id-C^*C$: let $S_C$ be the (unique) upper-triangular matrix with positive diagonal such that \[
				\id-C^*C = S_C^*S_C. \] It is easy to see that: 
				\[S_C = \begin{bmatrix}
					1 & 0 & 0 \\
					0 & \sqrt{1-|\omega_2|^2} & \frac{-\alpha_2 \cj{\omega_2}}{\sqrt{1 - |\omega_2|^2}} \\
					0 & 0 & \sqrt{\frac{1 - |\omega|^2 - |\alpha_2|^2}{1-|\omega_2|^2}}
				\end{bmatrix}.\]
				
				Moreover, for every $x \in \mathbb{C}^3$, we have $\left|\left|S_C \, x\right|\right|^2 = \ps{(\id-C^*C) \, x,  x} = \left|\left|D_C \, x\right|\right|^2$. Therefore, there exists an isometry $U_C \in \mathcal{B}\left(\mathbb{C}^3\right)$ such that $D_C=U_CS_C$. Since $U_C$ acts on the finite dimensional space $\mathbb{C}^3$, the operator $U_C$ is even unitary.
				
				Now, let $\widetilde{Z}=ZU_C$. We have
				\begin{equation}\label{eq-z-tilde-s-c}
					A=\widetilde{Z}S_C 
				\end{equation} 
				and 
				\begin{equation}\label{id-z-tilde-z^*}
					\id-ZZ^* = \id - \widetilde{Z}\widetilde{Z}^*.
				\end{equation}
				
				From \eqref{eq-z-tilde-s-c} we get:
				\[
				\widetilde{Z} = AS_C^{-1}= \begin{bmatrix}
					\omega_1 & \frac{\alpha_1}{\sqrt{1-|\omega_2|^2}} & \frac{\alpha_1\alpha2\cj{\omega_2} + \beta_1 (1-|\omega_2|^2)}{\sqrt{1-|\omega_2|^2}\sqrt{1-|\omega_2|-|\alpha_2|^2}}
				\end{bmatrix}.
				\]
				
				Finally, by \eqref{par-3_preuve_4_4} and \eqref{id-z-tilde-z^*}, $T$ is a contraction if and only if:
				
				\[
				\begin{multlined}\left| \gamma \left(1- |\omega_2|^2 - |\alpha_2|^2 \right) + \beta_2 \left(\cj{\omega_2} \alpha_1 + \cj{\alpha_2} \beta_1 \right) \right|^2 (1-|\omega_2|^2)  \\ \leq \left[ \left( (1-|\omega_1|^2)(1-|\omega_2|^2)-|\alpha_1|^2\right) \left(1 - |\omega_2|^2 - |\alpha_2|^2\right) - \left| \beta_1 (1-|\omega_2|^2) + \alpha_1 \alpha_2 \cj{\omega_2} \right|^2 \right]  \times \\ \left[ \left(1- |\omega_2|^2 - |\alpha_2|^2\right) \left(1 - |\omega_4|^2 - |\alpha_3|^2  \right) - |\beta_2|^2  \right]
				\end{multlined}.
				\]
				
					
					\item Assume now that $\omega_2 = 0$ and $|\alpha_2| <1$.
					
					We have then $\id-CC^* = \begin{bmatrix}
						1- |\alpha_2|^2 & 0 & 0 \\
						0 & 1 & 0 \\
						0 & 0 & 1
					\end{bmatrix}$ and $\id - C^*C = \begin{bmatrix}
						1 & 0 & 0 \\
						0 & 1 & 0 \\
						0 & 0 & 1-|\alpha_2|^2
					\end{bmatrix}$.
					
					Thus, we obtain \[
					Y= \left(\id-CC^*\right)^{-1/2} D = \begin{bmatrix}
						\frac{\beta_2}{\sqrt{1-|\alpha_2|^2}} \\
						\alpha_3 \\
						\omega_4
					\end{bmatrix} 
					\] and \[
					Z=A\left(\id-C^* C\right)^{-1/2} = \begin{bmatrix}
						\omega_1 & \alpha_1 & \frac{\beta_1}{\sqrt{1-|\alpha_2|^2}}\end{bmatrix}
					\]
					and we are brought back to the previous case (taking $\omega_2=0$ in the expressions of $Y$ and $Z$). Note that in this case, as $\id-C^*C$ is diagonal, $S_C=D_C$ and, thus, $\widetilde{Z}=Z$.

					\item Assume now that $\omega_2 \neq 0$ and $\alpha_2 = e^{i \theta_2} \sqrt{1-|\omega_2|^2}$, for $\theta_2 \in ]-\pi, \pi]$.
					
					Note that in that case, \eqref{contrac_4_4_deuxieme_ordre_proof} implies that $\beta_1 = \frac{- \alpha_1 \alpha_2  \cj{\omega_2}}{(1-|\omega_2|^2)}$ and $\beta_2 = 0$.
					
					Moreover, we have also:
					\[ \left(\id-CC^*\right)^{1/2}=\begin{bmatrix}
						0 & 0 & 0 \\
						0 & 1 & 0 \\
						0 & 0 & 1
					\end{bmatrix} \] and 
					\[
					\left(\id-C^*C\right)^{1/2} = \begin{bmatrix}
						1 & 0 & 0 \\
						0 & 1-|\omega_2|^2 & -e^{i \theta_2} \cj{\omega_2} \sqrt{1 - |\omega_2|^2} \\
						0 & - \omega_2 e^{-i \theta_2} \sqrt{1 - |\omega_2|} & |\omega_2|^2
					\end{bmatrix}.
					\]
					
					Let $Y= \begin{bmatrix}
						y_1 \\ y_2 \\ y_3 \end{bmatrix}$ and $Z= \begin{bmatrix}
						z_1 & z_2 & z_3
					\end{bmatrix}$.
					Then, $D = \left(\id - CC^*\right)^{1/2}Y$ if and only if $y_2 = \alpha_3$ and $y_3 = \omega_4$.
					Taking for $Y$ the minimal solution of the equation $D=D_{C^*}Y$, we can set $y_1=0$.
					
					Moreover, $A = Z \left(\id-C^* C\right)^{1/2}$ if and only if:
					\begin{align*}\left\lbrace \begin{matrix} z_1 & & & = \omega_1 \\
							& (1-|\omega_2|^2) z_2 & - e^{- i \theta_2} \omega_2 \sqrt{1-|\omega_2|^2} z_3 & = \alpha_1 \\
							&  - e^{i \theta_2} \cj{\omega_2} \sqrt{1- |\omega_2|^2} z_2 & + |\omega_2|^2 z_3 & = \beta_1
						\end{matrix}\right.
					\end{align*}
					which is equivalent to:
					\begin{align*}
						\left\lbrace \begin{matrix} z_1 & & & = \omega_1 \\
							& (1-|\omega_2|^2) z_2 & - e^{- i \theta_2} \omega_2 \sqrt{1-|\omega_2|^2} z_3 & = \alpha_1
						\end{matrix}\right.
					\end{align*}
					as we have, by assumption, $\beta_1 = \frac{- \alpha_1 e^{i \theta_2}  \cj{\omega_2}}{\sqrt{1-|\omega_2|^2}}$.
					Again, we take for $Z$ the minimal solution of the equation $A=ZD_C$, which means that $Z=0$ on $\text{Im}(D_C)^{\perp}$.
					
					It is easy to see that the rank of $D_C$ is equal to $2$, and that  $\text{Im}(D_C)^{\perp}$ is generated by $\begin{pmatrix} 0 \\ e^{i \theta_2} \cj{\omega_2} \\ \sqrt{1- |\omega_2|^2} \end{pmatrix}$. If we require that $D_C=0$ on $\text{Im}(D_C)^{\perp}$, then, we obtain:
					
					\begin{align*}
						\left\lbrace \begin{matrix} z_1 & = \omega_1 \\
							z_2 & = \alpha_1 \\
							z_3 & = \frac{- e^{i \theta_2} \cj{\omega_2} \alpha_1}{\sqrt{1-|\omega_2|^2}} \end{matrix}\right..
					\end{align*}
					Then we have $ZC^*Y =0$.
					
					Finally, $T$ is a contraction if and only if
					
					\[
					|\gamma |^2 \leq \left[1- |\omega_1|^2- \frac{|\alpha_1|^2}{1-|\omega_2|^2}\right] \left[1 -|\omega_4|^2 -|\alpha_3|^2 \right].
					\]

					\item Assume now that $\omega_2 =0$ and that $\alpha_2 = e^{i \theta_2}$, for some $\theta_2 \in ]- \pi, \pi]$.
				\end{itemize}
				
				We have then $\id-CC^* = \begin{bmatrix}
					0 & 0 & 0 \\
					0 & 1 & 0 \\
					0 & 0 & 1
				\end{bmatrix}$ and $\id - C^*C\begin{bmatrix}
					1 & 0 & 0 \\
					0 & 1 & 0 \\
					0 & 0 & 0
				\end{bmatrix}$.
				
				Again, let $Y= \begin{bmatrix}
					y_1 \\ y_2 \\ y_3 \end{bmatrix}$ and $Z= \begin{bmatrix}
					z_1 & z_2 & z_3
				\end{bmatrix}$.
				
				We have
				\[
				D = \left(\id - CC^*\right)^{1/2}Y \iff \systeme*{y_2 = \alpha_3,
					y_3 = \omega_4}.
				\]
				Taking for $Y$ the minimal solution of the equation $D=D_{C^*}Y$ gives us $y_1=0$.
				
				Moreover, we have $A = Z \left(\id-C^* C\right)^{1/2} \iff \systeme*{z_1 = \omega_1, z_2 = \alpha_1}$.
				
				Taking for $Z$ the minimal solution of the equation $A=ZD_C$ gives us $z_3=0$. We are therefore brought back to the previous case (taking $\omega_2=0$ in the expressions of $Y$ and $Z$).	
			\end{proof}
			
			\begin{defi}
				For $\omega \in \oud$, we define the Möbius transformation 
				$$M_\omega : \mathbb{C}\backslash\left\{\frac{1}{\cj{\omega}}\right\} \ni z \mapsto \frac{\omega-z}{1-\cj{\omega}z}$$
			\end{defi}
			
			We recall that $M_{\omega}$ is involutive, holomorphic on $\oud$, and that for all $z \in \cud$, $|M_{\omega}(z)| \leq 1$, with equality if and only if $|z|=1$. For more details, see for instance \autocite[Chapter IX, Section 2]{palkaIntroductionComplexFunction1991b}.
			
			\begin{defi}
				Let $\omega \in \oud$. For a Hilbert space $H$, we define 
				\[
				\mathcal{B}_ {\omega}(H) := \left\{ T \in \mathcal{B}(H) \: : \: \frac{1}{\cj{\omega}} \not \in \sigma(T) \right\}
				\]
				and \[
				M_{\omega}:= \mathcal{B}_ {\omega}(H) \ni T \mapsto (\omega \id - T)(\id - \cj{\omega}T)^{-1}.
				\]
			\end{defi}
			
			\begin{rema}
				All contractions belong to $\mathcal{B}_ {\omega}(H)$.
			\end{rema}
			
			\begin{lem}\label{lem-mobius-op} With the above notation:
				\begin{enumerate}
					\item $M_{\omega}$ is involutive ;
					\item $T$ is a contraction if and only if $M_{\omega}(T)$ is a contraction.
				\end{enumerate}
			\end{lem}
			
			\begin{proof}
				\leavevmode 
				\begin{enumerate}
					\item Firstly, it is worth noting that if $T \in \mathcal{B}_ {\omega}(H)$, then the spectrum $\sigma(M_{\omega}(T)) = M_{\omega}(\sigma(T))$ does not contain $\frac{1}{\cj{\omega}}$. This is because the equation $\frac{\omega-z}{1- \cj{\omega}z} = \frac{1}{\cj{\omega}}$ has no solution. Therefore, $M_{\omega}(M_{\omega}(T))$ is well-defined. Subsequently, using properties of the rational functional calculus, we can write $M_{\omega}(M_{\omega}(T)) = \left(M_{\omega} \circ M_{\omega} \right)(T) = T$.
					\item From the previous item, it is enough to prove that if $T$ is a contraction, then so is $M_{\omega}(T)$. This is proved in \autocite[page 14]{sz.-nagyHarmonicAnalysisOperators2010a}.
				\end{enumerate}
			\end{proof}
			
			\begin{proof}[Proof of \Cref{critere_4_4_gen}] \leavevmode
			First of all, we recall the reader the definition of \emph{divided differences} for pairwise distinct points $z_0, z_1, \cdots , z_n \in \mathbb{C}$ and $f : \mathbb{C} \to \mathbb{C}$ : the divided differences of $f$ at points $z_0, z_1, \cdots , z_n$ satisfy $[f(z_k)] = f(z_k)$ and the recurrence relation
						$$\left[f(z_k), \cdots , f(z_{k+j})\right] = \frac{\left[f(z_{k+1}), \cdots , f(z_{k+j})\right] - \left[f(z_{k}), \cdots , f(z_{k+j-1})\right] }{z_{k+j} - z_k},$$ for $ 0 \leq k \leq j \leq n$.
						We refer to \cite[Chapter 22]{simonLoewnersTheoremMonotone2019} or \autocite{deboorDividedDifferences2005b} for the definition and basic properties of \emph{divided differences} of $n+1$ (not necessarily distinct) points.
				\begin{itemize}[leftmargin=*]
					\item We first consider the case where $|\omega_2| < 1$ and $|\omega_3|<1$. \\
					In this case, we use \Cref{lem-mobius-op}: $T$ is a contraction if and only if $M_{\omega_3}(T)$ is a contraction. In the following, as long as there is no ambiguity, we will just write $M$ instead of $M_{\omega_3}$. Using the explicit rational functional calculus (see \eg \autocite{davisExplicitFunctionalCalculus1973b, simonLoewnersTheoremMonotone2019, highamFunctionsMatrices2008a}), we get the following matrix representation for $M(T)$:
					\begin{flalign}\label{matrice-M-t}
						&M(T)=\begin{pmatrix}
							M(\omega_1) & \alpha_1 [M(\omega_1), M(\omega_2)] & \lambda_1 & \mu \\
							0 & M(\omega_2) & \alpha_2 [M(\omega_2), M(\omega_3)] & \lambda_2 \\
							0 & 0 & M(\omega_3) & \alpha_3 [M(\omega_3), M(\omega_4)] \\
							0 & 0 & 0 & M(\omega_4)
						\end{pmatrix} &
					\end{flalign}
					\begin{flalign*}
						\text{where} \qquad &  \lambda_1 =  \alpha_1 \alpha_2 [M(\omega_1), M(\omega_2), M(\omega_3)]  + \beta_1 [M(\omega_1), M(\omega_3)] &  \\
						& \lambda_2  =  \alpha_2 \alpha_3 [M(\omega_2), M(\omega_3), M(\omega_4)]  + \beta_2 [M(\omega_2), M(\omega_4))] &   \\
						& \mu  =  \begin{multlined}[t] \alpha_1 \alpha_2 \alpha_3 [M(\omega_1), M(\omega_2), M(\omega_3), M(\omega_4)]  + \alpha_1 \beta_2 [M(\omega_1), M(\omega_2), M(\omega_4)] \\  + \alpha_3 \beta_1 [M(\omega_1), M(\omega_3), M(\omega_4)] + \gamma [M(\omega_1), M(\omega_4)]. \end{multlined} &
					\end{flalign*}
					
					For $i, j, k, l \in \llbracket 1, 4 \rrbracket$, we have:
					\begin{align*}
						[ M(\omega_i), M(\omega_j)] & = \frac{|\omega_3|^2-1}{(1-\cj{\omega_3}\omega_i)(1-\cj{\omega_3}\omega_j)} \\
						[M(\omega_i), M(\omega_j), M(\omega_k)] & = \frac{\cj{\omega_3}(|\omega_3|^2 -1)}{(1-\cj{\omega_3}\omega_i)(1-\cj{\omega_3}\omega_j)(1-\cj{\omega_3}\omega_k)} \\
						[M(\omega_i), M(\omega_j), M(\omega_k), M(\omega_l)] & = \frac{\cj{\omega_3}^2(|\omega_3|^2 -1)}{(1-\cj{\omega_3}\omega_i)(1-\cj{\omega_3}\omega_j)(1-\cj{\omega_3}\omega_k)(1-\cj{\omega_3}\omega_l)}
					\end{align*}
					Hence, \eqref{matrice-M-t} can be rewritten:
					\begin{align*}
						M(T)= \begin{pmatrix}
							\frac{\omega_3-\omega_1}{1-\cj{\omega_3}\omega_1} & \frac{\alpha_1 (|\omega_3|^2 -1 )}{(1-\cj{\omega_3}\omega_1)(1-\cj{\omega_3}\omega_2)} & \frac{-\alpha_1 \alpha_2 \cj{\omega_3} - \beta_1 (1- \cj{\omega_3}\omega_2)}{(1-\cj{\omega_3}\omega_1)(1-\cj{\omega_3}\omega_2)} & \mu \\
							0 & \frac{\omega_3-\omega_2}{1-\cj{\omega_3}\omega_2} & \frac{-\alpha_2}{1-\cj{\omega_3}\omega_2}  & \frac{-\alpha_2 \alpha_3 \cj{\omega_3} + \beta_2 (|\omega_3|^2 - 1)}{(1-\cj{\omega_3}\omega_2)(1- \cj{\omega_3}\omega_4)} \\
							0 & 0 & 0 & \frac{-\alpha_3}{1-\cj{\omega_3}\omega_4} \\
							0 & 0 & 0 & \frac{\omega_3-\omega_4}{1-\cj{\omega_3}\omega_4}
						\end{pmatrix}
					\end{align*}
					\begin{align*} \text{where} \quad  \mu & = \begin{multlined}[t] \frac{\gamma (|\omega_3|^2 -1)}{(1-\cj{\omega_3}\omega_1)(1-\cj{\omega_3}\omega_4)} + \frac{\alpha_1 \beta_2 \cj{\omega_3}(|\omega_3|^2 -1)}{(1-\cj{\omega_3}\omega_1)(1-\cj{\omega_3}\omega_2)(1-\cj{\omega_3}\omega_4)} \\ - \frac{\alpha_3 \beta_1 \cj{\omega_3}}{(1-\cj{\omega_3}\omega_1)(1-\cj{\omega_3}\omega_4)} - \frac{\alpha_1 \alpha_2 \alpha_3 \cj{\omega_3}^2}{(1-\cj{\omega_3}\omega_1)(1-\cj{\omega_3}\omega_2)(1-\cj{\omega_3}\omega_4)}
						\end{multlined} \\
						&= \frac{(|\omega_3|^2 -1)[\gamma (1-\cj{\omega_3}\omega_2)+ \alpha_1\beta_2\cj{\omega_3}]-\cj{\omega_3} [\alpha_3 \beta_1(1-\cj{\omega_3}\omega_2)+\alpha_1\alpha_2\alpha_3 \cj{\omega_3}]}{(1-\cj{\omega_3}\omega_1)(1-\cj{\omega_3}\omega_2)(1-\cj{\omega_3}\omega_4)}
					\end{align*}
					
					Now, we apply \Cref{crit-4-4-0} to $M(T)$, recalling that $|M(w_2)|=1$ if and only if $|\omega_2|=1$ and using the key identity:
					
					\begin{equation}\label{id-clef}
						\forall \: u , v \in \mathbb{C}, \: (1-|u|^2)(1-|v|^2) = |1 - \cj{u}v|^2 - |u - v |^2
					\end{equation}
					
					Using this identity, it is easy to see that 
					\[
					\left|\frac{-\alpha_2}{1-\cj{\omega_3}\omega_2}\right|^2 \leq 1 - \left| \frac{\omega_3 - \omega_2}{1 - \cj{\omega_3}\omega_2} \right|^2
					\]
					if and only if
					\[
					|\alpha_2|^2 \leq 1 - |\omega_2|^2,
					\]
					and that equality holds in one the two inequalities if and only if equality holds in the other one. 
					
					Thus, we can distinguish two subcases:
					
					\begin{itemize}[leftmargin=*]
						\item If $|\alpha_2|^2 < 1 - |\omega_2|^2$, it is easy to see (using \Cref{id-clef}) that \eqref{contrac_4_4_premier_ordre} and \eqref{contrac_4_4_deuxieme_ordre} applied to $M(T)$ are (respectively) equivalent to \eqref{contrac_4_4_premier_ordre_gen} and \eqref{contrac_4_4_deuxieme_ordre_gen}. Let us give some details to show that  \eqref{contrac_4_4_troisieme_ordre} applied to $M(T)$ is equivalent to \eqref{contrac_4_4_troisieme_ordre_gen}:
						
						\eqref{contrac_4_4_troisieme_ordre} applied to $M(T)$ is equivalent to 
						\begin{equation}\label{contrac_4_4_troisieme_ordre_intermediaire_lettres}
							\left|E\right|^2 \left(1-\left|\frac{\omega_3-\omega_2}{1-\cj{\omega_3}\omega_2} \right|^2 \right) \leq F \cdot G
						\end{equation}
						where
						\[
						E = \begin{multlined}[t]
							\frac{(|\omega_3|^2 -1)[\gamma (1- \cj{\omega_3}\omega_2)+\alpha_1\beta_2\cj{\omega_3}]-\cj{\omega_{3}}[\alpha_3\beta_1(1-\cj{\omega_3}\omega_2)+\alpha_1\alpha_2\alpha_3\cj{\omega_3}]}{(1-\cj{\omega_3}\omega_1)(1-\cj{\omega_3}\omega_2)(1-\cj{\omega_3}\omega_4)} \\ \times \left(1 - \left| \frac{\omega_3-\omega_2}{1-\cj{\omega_3}\omega_2} \right|^2 - \left| \frac{\alpha_2}{1-\cj{\omega_3}\omega_2}\right|^2\right) - \frac{\alpha_2 \alpha_3 \cj{\omega_3}+\beta_2(1-|\omega_3|^2)}{(1-\cj{\omega_3}\omega_2)(1-\cj{ \omega_3}\omega_4)} \\ \times \left( \frac{\alpha_1(\cj{ \omega_3} - \cj{ \omega_2})(|\omega_3|^2 -1)}{|1-\cj{ \omega_3} \omega_2|^2(1-\cj{ \omega_3}\omega_1)} + \frac{|\alpha_2|^2 \alpha_1 \cj{ \omega_3}+ \beta_1\cj{ \alpha_2} (1- \cj{ \omega_3} \omega_2)}{|1-\cj{ \omega_3} \omega_2|^2 (1-\cj{ \omega_3} \omega_1)}\right)
						\end{multlined}
						\]
						\[
						F = \begin{multlined}[t] \left[1 - \left| \frac{\omega_3 - \omega_2}{1- \cj{ \omega_3} \omega_2} \right|^2 -\left|\frac{\alpha_2}{1-\cj{ \omega_3}\omega_2} \right|^2 \right]\left[\left( 1 - \left| \frac{\omega_3 - \omega_1}{1- \cj{ \omega_3}\omega_1} \right|^2\right)\left( 1 - \left| \frac{\omega_3 - \omega_2}{1- \cj{ \omega_3}\omega_2} \right|^2\right) \right. \\ \left. - \left| \frac{\alpha_1 (|\omega_3|^2-1)}{(1-\cj{ \omega_3}\omega_1)(1-\cj{ \omega_3}\omega_2)} \right|^2  \right] - \left| \frac{-\alpha_1 \alpha_2 \cj{ \omega_3} - \beta_1 (1 - \cj{ \omega_3}\omega_2)}{(1-\cj{ \omega_3}\omega_1)(1- \cj{\omega_3}\omega_2)} \left( 1 - \left| \frac{\omega_3 - \omega_2}{1 - \cj{ \omega_3}\omega_2} \right|^2 \right) \right. \\ \left. + \frac{\alpha_1 \alpha_2(1-|\omega_3|^2)(\cj{ \omega_3} - \cj{ \omega_2})}{(1-\cj{ \omega_3}\omega_1)(1-\cj{ \omega_3} \omega_2)|1-\cj{ \omega_3}\omega_2|^2}\right|^2
						\end{multlined}
						\]
						and
						\begin{align*}
							G =  \begin{multlined}[t] \left(1 - \left| \frac{\omega_3 - \omega_2}{1 - \cj{ \omega_3}\omega_2} \right|^2 - \left| \frac{\alpha_2}{1- \cj{ \omega_3}\omega_2} \right|^2  \right) \left(1 - \left| \frac{\omega_3 - w_4}{1 - \cj{ \omega_3} \omega_4} \right|^2  - \left| \frac{\alpha_3}{1 - \cj{ \omega_3}\omega_4}\right|^2  \right)  \\ - \left| \frac{\alpha_2 \alpha_3 \cj{ \omega_3} + \beta_2 (1- |\omega_3|^2)}{(1-\cj{ \omega_3} \omega_2)(1- \cj{\omega_3}\omega_4)} \right|^2
							\end{multlined}
						\end{align*}
						
						On the one hand, we have:
						\begin{align*}
							& (1-\cj{ \omega_3}\omega_1)(1- \cj{ \omega_3}\omega_2)(1-\cj{ \omega_3}\omega_4) \left|1- \cj{ \omega_3}\omega_2\right|^2 \cdot E \\
							=&  \begin{multlined}[t]
								\left(|\omega_3|^2-1\right) (1-\cj{ \omega_3}\omega_2) \left[ \gamma \left( (1-|\omega_2|^2)(1-|\omega_3|^2) - |\alpha_2|^2 \right) + \alpha_1 \beta_2 \cj{ \omega_2} (1-|\omega_3|^2) \right. \\ \left. + \alpha_3 \beta_1 \cj{ \omega_3} (1-|\omega_2|^2) + \beta_1 \beta_2 \cj{ \alpha_2} + \alpha_1 \alpha_2 \alpha_3 \cj{\omega_2} \cj{ \omega_3} \right]
							\end{multlined}
						\end{align*}
						
						On the other hand, we have:
						\begin{align*}
							& \left|(1-\cj{ \omega_3} \omega_2)^2(1-\cj{ \omega_3}\omega_1)\right|^2 \cdot F \\
							= & \begin{multlined}[t] \left(1-|\omega_3|^2\right)^2 \left[\left( (1-|\omega_1|^2) (1-|\omega_2|^2) - |\alpha_1|^2 \right)\left( (1-|\omega_2|^2) (1-|\omega_3|^2) - |\alpha_2|^2 \right) \right. \\ \left. - \left| \alpha_1 \alpha_2 \cj{ \omega_2} + \beta_1 \left(1-|\omega_2|^2\right)\right|^2 \right]
							\end{multlined}
						\end{align*}
						and
						\begin{align*}
							& \left|(1-\cj {\omega_3}w_2)(1-\cj{ \omega_3}\omega_4)\right|^2 \cdot G \\
							= & \begin{multlined}[t]\left( (1-|\omega_2|^2) (1-|\omega_3|^2) - |\alpha_2|^2 \right)\left( (1-|\omega_3|^2) (1-|\omega_4|^2) - |\alpha_3|^2 \right) \\ - \left| \alpha_2 \alpha_3 \cj{ \omega_3} + \beta_2 \left(1-|\omega_3|^2\right)\right|^2
							\end{multlined}
						\end{align*}
						
						Therefore, \eqref{contrac_4_4_troisieme_ordre_intermediaire_lettres} is equivalent to \eqref{contrac_4_4_troisieme_ordre_gen}.
						
						\item If $|\alpha_2|^2 = 1 - |\omega_2|^2$, we check that \eqref{contrac_4_4_lim_cond_prem_ordre}, \eqref{contrac_4_4_lim_cond_deuxieme_ordre_b_1}, \eqref{contrac_4_4_lim_cond_deuxieme_ordre_b_2} and \eqref{contrac_4_4_lim_cond_trois_ordre} applied to $M(T)$ are equivalent to \eqref{contrac_4_4_lim_premier_ordre_gen}, \eqref{contrac_4_4_lim_deuxieme_ordre_gen} and \eqref{contrac_4_4_lim_troisieme_ordre_gen}.
					\end{itemize}
					
					%
					
					\item We are now considering the case where $|\omega_3|=1$.
					
					First of all, if $T$ is a contraction, then the compressions $\begin{bmatrix}
						\omega_1 & \alpha_1 & \beta_1 \\
						0 & \omega_2 & \alpha_2 \\
						0 & 0 & \omega_3
					\end{bmatrix}$ and $\begin{bmatrix}
						\omega_2 & \alpha_2 & \beta_2 \\
						0 & \omega_3 & \alpha_3 \\
						0 & 0 & \omega_4 
					\end{bmatrix}$ are contractions. As it is a necessary condition for $T$ to be a contraction, we will assume in the following that it is the case. 
					
					let $A=\begin{bmatrix}
						\omega_1 & \alpha_1 & \beta_1
					\end{bmatrix}$, $B = \begin{bmatrix} \gamma \end{bmatrix}$, $C=\begin{bmatrix}
						0 & \omega_2 & \alpha_2 \\ 0 & 0 & \omega_3 \\ 0 & 0 & 0
					\end{bmatrix}$ and $D=\begin{bmatrix}
						\beta_2 \\ \alpha_3 \\ \omega_4
					\end{bmatrix}$.
					
					By assumption, $\begin{bmatrix}
						A \\ C
					\end{bmatrix}$ and $\begin{bmatrix}
						C & D
					\end{bmatrix}$ are contractions.

					By Parrott's theorem, $T$ is a contraction if and only if:
					\begin{equation}{} 
						B = (\id-ZZ^*)^{1/2}V(\id-Y^*Y)^{1/2}-ZC^*Y, \text{ for some contraction } V,  \label{parrott-eq} \end{equation}
					where $Y$ and $Z$ are contractions such that $D=(\id-CC^*)^{1/2}Y$ and $A=Z(\id-C^*C)^{1/2}$.
					
						
						We have:
						$(\id - CC^*)^{1/2} = \begin{bmatrix}
							\sqrt{1-|\omega_2|^2} & 0 & 0 \\
							0 & 0 & 0 \\
							0 & 0 & 1
						\end{bmatrix}$ \quad ; \quad $(\id - C^* C)^{1/2} = \begin{bmatrix} 1 & 0 & 0 \\
							0 & \sqrt{1-|\omega_2|^2} & 0 \\
							0 & 0 & 0 \end{bmatrix}$
						
						\begin{itemize}[leftmargin=*]
							\item If $|\omega_2| < 1$:
							
							Applying the criterion for $3 \times 3$ matrices, the assumption that the compressions $\begin{bmatrix}
								\omega_1 & \alpha_1 & \beta_1 \\
								0 & \omega_2 & \alpha_2 \\
								0 & 0 & \omega_3
							\end{bmatrix}$ and $\begin{bmatrix}
								\omega_2 & \alpha_2 & \beta_2 \\
								0 & \omega_3 & \alpha_3 \\
								0 & 0 & \omega_4 
							\end{bmatrix}$ are contractions means that $|\alpha_1|^2 \leq (1-|\omega_1|^2)(1-|\omega_2|^2)$, $\alpha_2 = \alpha_3 =0$, $|\beta_1(1-|\omega_2|^2)|^2 \leq (1-|\omega_1|^2)(1-|\omega_2|^2)-|\alpha_1|^2$ and $|\beta_2|^2 \leq (1-|\omega_2|^2)(1-|\omega_4|^2)$.
							
							Moreover, the equations $D=(\id-CC^*)^{1/2}Y$ and $A=Z(\id-C^*C)^{1/2}$ have both infinitely many solutions. Taking the minimal solutions for $Y$ and $Z$, we obtain $Y = \begin{bmatrix} \frac{\beta_2}{\sqrt{1-|\omega_2|^2}} \\ 0 \\ \omega_4 \end{bmatrix}$ and $Z = \begin{bmatrix}
								\omega_1 & \frac{\alpha_1}{\sqrt{1-|\omega_2|^2}} & 0
							\end{bmatrix}$
							
							Hence, we have $ZC^* Y = \frac{\alpha_1 \beta_2 \cj{\omega_2}}{1-|\omega_2|^2}$ and, in the end, $T$ is a contraction if and only if: 
							\[
							\left|\gamma + \frac{\alpha_1 \beta_2 \cj{\omega_2}}{1-|\omega_2|^2} \right|^2 \leq \left[ 1 - |\omega_4|^2 - \frac{|\beta_2|^2}{1- |\omega_2|^2}  \right] \left[ 1 - |\omega_1|^2 - \frac{|\alpha_1|^2}{1- |\omega_2|^2} \right],
							\]
							which is equivalent to
							\[
							\left|\gamma (1-|\omega_2|^2) + \alpha_1 \beta_2 \cj{\omega_2} \right|^2 \leq \left[ (1 - |\omega_4|^2)(1-|\omega_2|^2) - |\beta_2|^2  \right] \left[ (1 - |\omega_1|^2)(1-|\omega_2|^2) - |\alpha_1|^2 \right].
							\]
							
							\item If $|\omega_2| = 1$: 
							
							Applying the criterion for $3 \times 3$ matrices, the assumption that the compressions $\begin{bmatrix}
								\omega_1 & \alpha_1 & \beta_1 \\
								0 & \omega_2 & \alpha_2 \\
								0 & 0 & \omega_3
							\end{bmatrix}$ and $\begin{bmatrix}
								\omega_2 & \alpha_2 & \beta_2 \\
								0 & \omega_3 & \alpha_3 \\
								0 & 0 & \omega_4 
							\end{bmatrix}$ are contractions means that $\alpha_1 = \alpha_2 = \alpha_3 =0$, $|\beta_1|^2 \leq (1-|\omega_1|^2)(1-|\omega_3|^2)$ and $|\beta_2|^2 \leq (1-|\omega_2|^2)(1-|\omega_4|^2)$.
							Moreover, taking the minimal solutions for $Y$ and $Z$, we obtain $Y = \begin{bmatrix} 0 \\ 0 \\ \omega_4 \end{bmatrix}$ and $Z = \begin{bmatrix}
								\omega_1 & 0 & 0
							\end{bmatrix}$
							Hence, we have $ZC^* Y = 0$ and, in the end, $T$ is a contraction if and only if: 
							\begin{align*}
								\left|\gamma \right|^2 \leq (1-|\omega_1|^2)(1-|\omega_4|^2)
							\end{align*}
						\end{itemize}
						\item The case where $|\omega_2|=1$ and $|\omega_3|<1$ is similar to the case where $|\omega_3|=1$ and $|\omega_2|<1$.
					\end{itemize}
				\end{proof}

\section*{Acknowledgments}
The author thanks Catalin Badea for useful comments and encouragements. 
The author also acknowledge support from the Labex CEMPI (ANR-11-LABX-0007-01).

\printbibliography

\end{document}